\documentclass{article}
\usepackage{amssymb}
\usepackage{amsmath}
\usepackage{amsfonts}
\usepackage{everypage}
\usepackage{draftwatermark}
\usepackage{sw20bams}

\newtheorem{theorem}{Theorem}
\newtheorem{acknowledgement}{Acknowledgement}

\newtheorem{conclusion}[theorem]{Conclusion}

\newtheorem{corollary}[theorem]{Corollary}

\newtheorem{definition}[theorem]{Definition}
\newtheorem{example}[theorem]{Example}

\newtheorem{lemma}[theorem]{Lemma}

\newtheorem{proposition}[theorem]{Proposition}
\newtheorem{remark}[theorem]{Remark}

\input{tcilatex}
\SetWatermarkText{Accepted for publication in New Math and Nat. Comp.}

\begin{document}

\title{Distance and Similarity Measures for Soft Sets\bigskip }
\author{Athar Kharal\thanks{%
Corresponding author. Phone: +92 333 6261309} \thanks{%
This paper was submitted to NM\&NC on 28 January 2009 and was accepted for
publication on 16 June 2009. It will apear in print in (hopefully) Nov 2010
issue of NM\&NC.}\bigskip  \\
College of Aeronautical Engineering,\\
National University of Sciences and Technology (NUST),\\
PAF Academy Risalpur 24090, PAKISTAN\\
\textit{atharkharal@gmail.com}}
\date{ }
\maketitle

\begin{abstract}
In [P. Majumdar, S. K. Samanta, \textit{Similarity measure of soft sets},
New Mathematics and Natural Computation 4(1)(2008) 1-12], the authors use
matrix representation based distances of soft sets to introduce matching
function and distance based similarity measures. We first give
counterexamples to show that their Definition 2.7 and Lemma 3.5(3) contain
errors, then improve their Lemma 4.4 making it a corllary of our result. The
fundamental assumption of \cite{sfst08maj} has been shown to be flawed. This
motivates us to introduce set operations based measures. We present a case
(Example \ref{ex_showing_superiority}) where Majumdar-Samanta similarity
measure produces an erroneous result but the measure proposed herein decides
correctly. Several properties of the new measures have been presented and
finally the new similarity measures have been applied to the problem of
financial diagnosis of firms.\bigskip

\textbf{Keywords:} Applied soft sets; Similarity measure; Distance measure;
Financial diagnosis; Similarity based decision making;
\end{abstract}

\section{\textbf{Introduction\protect\bigskip }}

In 1999, D. Molodtsov \cite{sfst99mol}, introduced the notion of a soft set
as a collection of approximate descriptions of an object. This initial
description of the object has an approximate nature, and we do not need to
introduce the notion of exact solution. The absence of any restrictions on
the approximate description in soft sets make this theory very convenient
and easily applicable in practice. Applications of soft sets in areas
ranging from decision problems to texture classification, have surged in
recent years \cite{sfst07kov,sfst02maj,sfst06mus,sfst05xia,sfst08zou}.

Similarity measures quantify the extent to which different patterns,
signals, images or sets are alike. Such measures are used extensively in the
application of fuzzy sets, intuitionistic fuzzy set and vague sets to the
problems of pattern recognition, signal detection, medical diagnosis and
security verification systems. That is why several researchers have studied
the problem of similarity measurement between fuzzy sets \cite{fst93hyu},
intuitionistic fuzzy sets (IFSs) and vague sets. Ground breaking work for
introducing similarity measure of soft sets was presented by Majumdar and
Samanta in \cite{sfst08maj}. Their work uses matrix representation based
distances of soft sets to introduce similarity measures. In this paper, we
propose new similarity measures using set theoretic operations, besides
showing how the earlier similarity measures of Majumdar and Samanta are
inappropriate. We also present an application of the proposed measures of
similarity in the area of automated financial analysis.

This paper is organized as follows: in Section 2, requisite preliminary
notions from Soft Set Theory have been presented. Section 3 comprises some
counterexamples to show that some claims in \cite{sfst08maj} are not
correct. At the end of this section we also improve and generalize Lemma 4.4
of \cite{sfst08maj}. In Section 4 we give our motivation and rationale to
introduce set operations based distance and similarity measures. Section 5
introduces the notion of set operation based distance between soft sets and
some of its weaker forms. In Section 6 similarity measures have been
defined. Finally Section 7 is the application of new similarity measures to
the problem of financial diagnosis of firms.\bigskip

\section{\textbf{Preliminaries\protect\bigskip }}

A pair $(F,A)$ is called a soft set \cite{sfst99mol} over $X$, where $F$ is
a mapping given by $F:A\rightarrow P(X).$\newline
In other words, a soft set over $X$ is a parametrized family of subsets of
the universe $X.$ For $\varepsilon \in A,$ $F(\varepsilon )$ may be
considered as the set of $\varepsilon $-approximate elements of the soft set 
$(F,A)$. Clearly a soft set is not a set in ordinary sense.

\begin{definition}
\cite{submission-OnSfSts} Let $X$ be a universe and $E$ a set of attributes.
Then the pair $\left( X,E\right) ,$ called a soft space, is the collection
of all soft sets on $X$ with attributes from $E$.
\end{definition}

\begin{definition}
\cite{sfst05pei}\label{st-subset} For two soft sets $(F,A)$ and $(G,B)$ over 
$X$, we say that $(F,A)$ is a soft subset of $(G,B),$ if\newline
$(i)$ $A\subseteq B,$ and\newline
$(ii)$ $\forall ~\varepsilon \in A,F(\varepsilon )\subseteq G(\varepsilon )$.%
\newline
We write $(F,A)~\widetilde{\subseteq }~(G,B)$. $(F,A)$ is said to be a soft
super set of $(G,B)$, if $(G,B)$ is a soft subset of $(F,A)$. We denote it
by $(F,A)~\widetilde{\supseteq }~(G,B)$.
\end{definition}

\begin{definition}
\cite{sfst03maj}\label{st-Union} Union of two soft sets $(F,A)$ and $(G,B)$
over a common universe $X$ is the soft set $(H,C),$ where $C=A\cup B,$ and $%
\forall ~\varepsilon \in C,$%
\begin{equation*}
H(\varepsilon )=\left\{ 
\begin{array}{ccc}
F(\varepsilon ), & if & \varepsilon \in A-B, \\ 
G(\varepsilon ), & if & \varepsilon \in B-A, \\ 
F(\varepsilon )\cup G(\varepsilon ), & if & \varepsilon \in A\cap B.%
\end{array}%
\right.
\end{equation*}%
We write $(F,A)~\tilde{\cup}~(G,B)=(H,C).$
\end{definition}

\begin{definition}
\label{intersection_def_Ours}\cite{sfst08ali} Let $(F,A)$ and $(G,B)$ be two
soft sets over $X$ with $A\cap B\neq \phi $. Restricted intersection of two
soft sets $(F,A)$ and $(G,B)$ is a soft set $(H,C),$ where $C=A\cap B,$ and $%
\forall ~\varepsilon \in C,H(\varepsilon )=F(\varepsilon )\cap G(\varepsilon
)$. We write $(F,A)\tilde{\cap}(G,B)=(H,C).$
\end{definition}

\begin{definition}
\cite{sfst08maj} The complement of a soft set $\left( F,A\right) $ is
denoted by $\left( F,A\right) ^{c},$ and is defined by $\left( F,A\right)
^{c}=\left( F^{c},A\right) ,$ where $F^{c}:A\rightarrow P\left( X\right) $
is a mapping given by $F^{c}\left( \varepsilon \right) =\left( F\left(
\varepsilon \right) \right) ^{c},\forall \varepsilon \in A.$
\end{definition}

In the sequel we shall denote the absolute null and absolute whole soft sets
in a soft space $\left( X,E\right) $ as $\left( F_{\phi },E\right) $ and $%
\left( F_{X},E\right) $ , respectively. These have been defined in \cite%
{sfst08ali} as:%
\begin{eqnarray*}
\left( F_{\phi },E\right) &=&\left\{ \varepsilon =\phi ~|~\forall
\varepsilon \in E\right\} , \\
\left( F_{X},E\right) &=&\left\{ \varepsilon =X~|~\forall \varepsilon \in
E\right\} .
\end{eqnarray*}

\section{\textbf{Counterexamples\protect\bigskip }}

Recently Majumdar and Samanta \cite{sfst08maj} have written the ground
breaking paper on similarity measures of soft sets. In this section we first
give counterexamples to show that their Definition 2.7 and Lemma 3.5(3)
contain some errors. We, then, improve Lemma 4.4 \cite{sfst08maj} and make
it a corollary of our result.

A matching function based similarity measure has been defined in \cite%
{sfst08maj} as:

\begin{definition}
\cite{sfst08maj} If $E_{1}=E_{2},$ then similarity between $\left(
F_{1},E_{1}\right) $ and $\left( F_{2},E_{2}\right) $ is defined by%
\begin{equation}
S\left( F_{1},F_{2}\right) =\frac{\Sigma _{i}\overrightarrow{F}_{1}\left(
e_{i}\right) \cdot \overrightarrow{F}_{2}\left( e_{i}\right) }{\Sigma _{i}%
\left[ \overrightarrow{F}_{1}\left( e_{i}\right) ^{2}\vee \overrightarrow{F}%
_{2}\left( e_{i}\right) ^{2}\right] }.  \tag{I}  \label{MtchFnSim}
\end{equation}%
If $E_{1}\neq E_{2}$ and $E=E_{1}\cap E_{2}\neq \phi ,$ then we first define 
$\overrightarrow{F}_{1}\left( e\right) =\underline{0}$ for $e\in
E_{2}\backslash E$ and $\overrightarrow{F}_{2}\left( f\right) =\underline{0}$
for $f\in E_{1}\backslash E.$ Then $S\left( F_{1},F_{2}\right) $ is defined
by formula (\ref{MtchFnSim}).
\end{definition}

\begin{lemma}
\label{lm3.5,sfst08maj}(Lemma 3.5 \cite{sfst08maj}) Let $\left(
F_{1},E_{1}\right) $ and $\left( F_{2},E_{2}\right) $ be two soft sets over
the same finite universe $U.$ Then the following hold:

$\left( 1\right) $ $S\left( F_{1},F_{2}\right) =S\left( F_{2},F_{1}\right) $

$\left( 2\right) $ $0\leq S\left( F_{1},F_{2}\right) \leq 1$

$\left( 3\right) $ $S\left( F_{1},F_{1}\right) =1$
\end{lemma}

Our next example shows that claim $\left( 3\right) $ of Lemma \ref%
{lm3.5,sfst08maj} is incorrect:

\begin{example}
Let $X=\left\{ a,b,c\right\} $ and $E=\left\{ e_{1},e_{2},e_{3}\right\} .$
We choose soft set $\left( F_{1},E\right) $ as:%
\begin{equation*}
\left( F_{1},E\right) =\left\{ e_{1}=\left\{ {}\right\} ,e_{2}=\left\{
{}\right\} ,e_{3}=\left\{ {}\right\} \right\}
\end{equation*}%
Then using (\ref{MtchFnSim}) we get%
\begin{equation*}
S\left( F_{1},F_{1}\right) =\frac{0}{0}.
\end{equation*}
\end{example}

Majumdar and Samanta have defined following distances between soft sets as
four distinct notions:

\begin{definition}
\cite{sfst08maj} For two soft sets $\left( \widehat{F}_{1},E\right) $ and $%
\left( \widehat{F}_{2},E\right) $ we define the mean Hamming distance $%
D^{s}\left( \widehat{F}_{1},\widehat{F}_{2}\right) $ between soft sets as:%
\begin{equation*}
D^{s}\left( \widehat{F}_{1},\widehat{F}_{2}\right) =\frac{1}{m}\left\{
\dsum\limits_{i=1}^{m}\dsum\limits_{j=1}^{n}\left\vert \widehat{F}_{1}\left(
e_{i}\right) \left( x_{j}\right) -\widehat{F}_{2}\left( e_{i}\right) \left(
x_{j}\right) \right\vert \right\} ,
\end{equation*}%
the normalized Hamming distance $L^{s}\left( \widehat{F}_{1},\widehat{F}%
_{2}\right) $ as:%
\begin{equation*}
L^{s}\left( \widehat{F}_{1},\widehat{F}_{2}\right) =\frac{1}{m~n}\left\{
\dsum\limits_{i=1}^{m}\dsum\limits_{j=1}^{n}\left\vert \widehat{F}_{1}\left(
e_{i}\right) \left( x_{j}\right) -\widehat{F}_{2}\left( e_{i}\right) \left(
x_{j}\right) \right\vert \right\} ,
\end{equation*}%
the Euclidean distance $E^{s}\left( \widehat{F}_{1},\widehat{F}_{2}\right) $
as:%
\begin{equation*}
E^{s}\left( \widehat{F}_{1},\widehat{F}_{2}\right) =\sqrt{\frac{1}{m}%
\dsum\limits_{i=1}^{m}\dsum\limits_{j=1}^{n}\left( \widehat{F}_{1}\left(
e_{i}\right) \left( x_{j}\right) -\widehat{F}_{2}\left( e_{i}\right) \left(
x_{j}\right) \right) ^{2}},
\end{equation*}%
the normalized Euclidean distance as:%
\begin{equation*}
Q^{s}\left( \widehat{F}_{1},\widehat{F}_{2}\right) =\sqrt{\frac{1}{m~n}%
\dsum\limits_{i=1}^{m}\dsum\limits_{j=1}^{n}\left( \widehat{F}_{1}\left(
e_{i}\right) \left( x_{j}\right) -\widehat{F}_{2}\left( e_{i}\right) \left(
x_{j}\right) \right) ^{2}}.
\end{equation*}
\end{definition}

\begin{remark}
\label{mainRemark}Majumdar and Samanta's observation (also used in their
proof of Lemma 4.4 \cite{sfst08maj}) that 
\begin{equation*}
\left\vert F\left( e_{i}\right) \left( x_{j}\right) -G\left( e_{i}\right)
\left( x_{j}\right) \right\vert \leq 1
\end{equation*}%
is inaccurate. The quantity $\left\vert F\left( e_{i}\right) \left(
x_{j}\right) -G\left( e_{i}\right) \left( x_{j}\right) \right\vert $ is
either $0$ or $1,$ only. Consequently, the term $\left( F\left( e_{i}\right)
\left( x_{j}\right) -G\left( e_{i}\right) \left( x_{j}\right) \right) ^{2},$
used for defining distances $E^{s}$ and $Q^{s}$, comes out to be identical
with 
\begin{equation*}
\left\vert F\left( e_{i}\right) \left( x_{j}\right) -G\left( e_{i}\right)
\left( x_{j}\right) \right\vert .
\end{equation*}%
This renders $E^{s}$ and $Q^{s}$ as mere square roots of $D^{s}$ and $L^{s},$
respectively. Symbolically we write:%
\begin{equation*}
E^{s}=\sqrt{D^{s}}\text{ and }Q^{s}=\sqrt{L^{s}}.
\end{equation*}%
Hence the four distances of Majumdar and Samanta are not distinct, rather
they are only two distances.
\end{remark}

In the sequel the cardinality of a set $A$ is denoted as $\left\Vert
A\right\Vert $. We now present the main result of this section as:

\begin{theorem}
\label{mainResult}Let $m=\left\Vert E\right\Vert ,n=\left\Vert X\right\Vert $%
. Then for any two soft sets $\left( F_{1},E\right) $ and $\left(
F_{2},E\right) $ we have\medskip

$\left( 1\right) $ $D^{s}\left( F_{1},F_{2}\right) \in \left\{ \frac{k}{m}%
~|~k=0,1,2,...,mn\right\} ,\bigskip $

$\left( 2\right) $ $L^{s}\left( F_{1},F_{2}\right) \in \left\{ \frac{k}{mn}%
~|~k=0,1,2,...,mn\right\} ,\bigskip $

$\left( 3\right) $ $E^{s}\left( F_{1},F_{2}\right) \in \left\{ \sqrt{\frac{k%
}{m}}~|~k=0,1,2,...,mn\right\} ,\smallskip $

$\left( 4\right) $ $Q^{s}\left( F_{1},F_{2}\right) \in \left\{ \sqrt{\frac{k%
}{mn}}~|~k=0,1,2,...,mn\right\} .$
\end{theorem}

\begin{proof}
$\left( 1\right) $ The smallest and the largest distances are given as%
\begin{eqnarray}
D^{s}\left( \left( F_{\phi },E\right) ,\left( F_{\phi },E\right) \right) &=&%
\frac{1}{m}\left\{ \dsum\limits_{i=1}^{m}\dsum\limits_{j=1}^{n}\left\vert
F_{\phi }\left( e_{i}\right) \left( x_{j}\right) -F_{\phi }\left(
e_{i}\right) \left( x_{j}\right) \right\vert \right\} =0.  \TCItag{II}
\label{II} \\
D^{s}\left( \left( F_{\phi },E\right) ,\left( F_{X},E\right) \right) &=&%
\frac{1}{m}\left\{ \dsum\limits_{i=1}^{m}\dsum\limits_{j=1}^{n}\left\vert
F_{\phi }\left( e_{i}\right) \left( x_{j}\right) -F_{X}\left( e_{i}\right)
\left( x_{j}\right) \right\vert \right\}  \notag \\
&=&\frac{1}{m}\left\{ \dsum\limits_{i=1}^{m}\dsum\limits_{j=1}^{n}\left\vert
0-1\right\vert \right\} ,~\text{since }F_{\phi }\left( e_{i}\right) \left(
x_{j}\right) =0~\text{and }F_{X}\left( e_{i}\right) \left( x_{j}\right)
=1~\forall i,j.  \notag \\
&=&\frac{1}{m}\underset{mn\text{ times}}{\underbrace{\left( \left\vert
0-1\right\vert +\left\vert 0-1\right\vert +...+\left\vert 0-1\right\vert
\right) }}=\frac{mn}{m}=n.  \TCItag{III}  \label{III}
\end{eqnarray}%
Furthermore, suppose the arrangement of entries in matrix representation of
two arbitrary soft sets $\left( F_{1},E\right) $ and $\left( F_{2},E\right) $
is such that the term $\left\vert F_{1}\left( e_{i}\right) \left(
x_{j}\right) -F_{2}\left( e_{i}\right) \left( x_{j}\right) \right\vert $
evaluates to $1,$ $k$ times. Then, we can re-arrange the terms in expansion
of $D^{s}\left( F_{1},F_{2}\right) $ to get%
\begin{eqnarray}
D^{s}\left( F_{1},F_{2}\right) &=&\frac{1}{m}\left( \underset{k\text{ times}}%
{\underbrace{\left( \left\vert 0-1\right\vert +\left\vert 1-0\right\vert
+\left\vert 0-1\right\vert +...+\left\vert 1-0\right\vert \right) }}\right. +
\notag \\
&&\left. \underset{mn-k\text{ times}}{\underbrace{\left( \left\vert
0-0\right\vert +\left\vert 1-1\right\vert +\left\vert 0-0\right\vert
+...+\left\vert 1-1\right\vert \right) }}\right)  \notag \\
&=&\frac{1}{m}\left( k+0\right) =\frac{k}{m}.  \TCItag{IV}  \label{IV}
\end{eqnarray}%
By (\ref{II}),(\ref{III}) and (\ref{IV}) we have $D^{s}\left(
F_{1},F_{2}\right) \in \left\{ \frac{k}{m}~|~k=0,1,2,...,mn\right\}
.\bigskip $

$\left( 2\right) $ Note that $L^{s}\left( F_{1},F_{2}\right) =\frac{1}{n}%
D^{s}\left( F_{1},F_{2}\right) .$ The result now follows immediately from $%
\left( 1\right) .\bigskip $

$\left( 3\right) ,\left( 4\right) $ Follow immediately by Remark \ref%
{mainRemark} and $\left( 1\right) $ and $\left( 2\right) .$
\end{proof}

\begin{corollary}
(Lemma 4.4 \cite{sfst08maj}) Let $m=\left\Vert E\right\Vert ,n=\left\Vert
X\right\Vert $. Then for any two soft sets $\left( F_{1},E\right) $ and $%
\left( F_{2},E\right) $ we have

$\left( 1\right) $ $D^{s}\left( F_{1},F_{2}\right) \leq n,$

$\left( 2\right) $ $L^{s}\left( F_{1},F_{2}\right) \leq 1,$

$\left( 3\right) $ $E^{s}\left( F_{1},F_{2}\right) \leq \sqrt{n},$

$\left( 4\right) $ $Q^{s}\left( F_{1},F_{2}\right) \leq 1$\bigskip .
\end{corollary}

\section{\textbf{Motivation for Introducing New Distance and Similarity
Measure\protect\bigskip s}}

We first define the notion of soft space:

\begin{definition}
Let $X$ be a universe and $E$ a set of attributes. Then the pair $\left(
X,E\right) ,$ called a soft space, is the collection of all soft sets on $X$
with attributes from $E$.
\end{definition}

The work of Majumdar and Samanta depends solely upon the tacit assumption
that matrix representation of soft sets is a suitable representation. We now
discuss the validity of this assumption.

Tabular representation of a soft set was first proposed by Maji, Biswas and
Roy in \cite{sfst03maj}. This representation readily lends itself to become
Majumdar and Samanta's matrix representation as given in \cite{sfst08maj}.
Hence, in the sequal, we shall use the words `table representation' and
`matrix representation' interchangeably. Furthermore, we shall term a soft
set in a soft space $\left( X,E\right) $ as `total soft set' if the soft
set, which is a mapping, is defined on each point of the universe of
attributes $E.$ Hence $\left( F,E\right) $ is a total soft set in the soft
space $\left( X,E\right) ,$ but $\left( G,B\right) $, with $B\subset E,$ is
not.

It is noteworthy that the matrix representation compels one to write every
soft set as a total soft set. Consequently neither matrix representation is
unique, nor it returns the original soft set. This is shown by the following
example:

\begin{example}
Let $\left( X,E\right) $ be a soft space with $X=\left\{ a,b,c\right\} $ and 
$E=\left\{ e_{1},e_{2},e_{3}\right\} .$ Choose%
\begin{equation*}
\left( F,A\right) =\left\{ e_{1}=\left\{ a,c\right\} ,e_{3}=\left\{
b,c\right\} \right\} ,
\end{equation*}%
then its matrix representation is given as%
\begin{equation*}
F^{\prime }=\left( 
\begin{array}{ccc}
1 & 0 & 0 \\ 
0 & 0 & 1 \\ 
1 & 0 & 1%
\end{array}%
\right) .
\end{equation*}%
If we try to retrieve, the soft set $\left( F,A\right) $ from $F^{\prime }$
we get%
\begin{equation*}
\left( G,E\right) =\left\{ e_{1}=\left\{ a,c\right\} ,e_{2}=\left\{
{}\right\} ,e_{3}=\left\{ b,c\right\} \right\} ,
\end{equation*}%
which is clearly a different soft set in $\left( X,E\right) $ as $\left(
F,A\right) ~\widetilde{\subseteq }~\left( G,E\right) $ but $\left(
G,E\right) ~\widetilde{\not\subseteq }~\left( F,A\right) $ and hence $\left(
F,A\right) \neq \left( G,E\right) .$
\end{example}

Moreover, it is evident by the very definition of soft union as given by
Maji \textit{et.al.} that total soft sets are not meant by either Molodstov 
\cite{sfst99mol} or Maji, Biswas and Roy \cite{sfst03maj}. Had this been the
case, the soft union should not have been defined in three pieces.

Furthermore, it is important to note that in \cite{sfst08maj}, while
calculating the similarity only the value sets of a soft set have been paid
attention to. Whereas, ideally a similarity measure for soft sets must
reflect similarity between both the value sets and the attributes, due to
the peculiar dependance of the notion of soft set upon these two sets.

Both the above given points viz. non-suitability of matrix representation
and partial nature of similarity measures of Majumdar and Samanta, provide
us motivation to introduce more suitable distance and similarity measures of
soft sets. We introduce these measures in the following sections.\bigskip

\section{\textbf{Distance Between Soft Sets\protect\bigskip }}

Recall that symmetric difference between two sets $A$ and $B$ is denoted and
defined as:%
\begin{equation*}
A\Delta B=\left( A\cup B\right) \backslash \left( A\cap B\right) .
\end{equation*}%
We first define:

\begin{definition}
Let $\left( F,A\right) ,\left( G,B\right) $ and $\left( H,C\right) $ be soft
sets in a soft space $\left( X,E\right) $ and $d:X\times X\rightarrow 
\mathbb{R}^{\mathbb{+}}$ a mapping. Then

$\left( 1\right) $ $d$ is said to be quasi-metric if it satisfies

~$~~~~~~~~~~~~~\left( M_{1}\right) ~d\left( \left( F,A\right) ,\left(
G,B\right) \right) \geq 0,$

~$~~~~~~~~~~~~~\left( M_{2}\right) ~d\left( \left( F,A\right) ,\left(
G,B\right) \right) =d\left( \left( G,B\right) ,\left( F,A\right) \right) ,$

$\left( 2\right) $ A quasi-metric $d$ is said to be semi-metric if

~$~~~~~~~~~~~~~\left( M_{3}\right) ~d\left( \left( F,A\right) ,\left(
G,B\right) \right) +d\left( \left( G,B\right) ,\left( H,C\right) \right)
\geq d\left( \left( F,A\right) ,\left( H,C\right) \right) $

$\left( 3\right) $ A semi-metric $d$ is said to be pseudo metric if

~$~~~~~~~~~~~~~\left( M_{4}\right) ~\left( F,A\right) =\left( G,B\right)
\Rightarrow d\left( \left( F,A\right) ,\left( G,B\right) \right) =0,$

$\left( 4\right) $ A pseudo metric $d$ is said to be metric if

~$~~~~~~~~~~~~~\left( M_{5}\right) ~d\left( \left( F,A\right) ,\left(
G,B\right) \right) =0\Rightarrow \left( F,A\right) =\left( G,B\right) .$
\end{definition}

Some quasi-metrics and semi-metrics for soft sets may readily be defined as
follows:

\begin{definition}
For two soft sets $\left( F,A\right) $ and $\left( G,B\right) $ in a soft
space $\left( X,E\right) ,$ where $A$ and $B$ are not identically void$,$ we
define Hamming quasi-metric as:%
\begin{equation*}
d\left( \left( F,A\right) ,\left( G,B\right) \right) =\left\Vert A\Delta
B\right\Vert +\sum_{\varepsilon \in A\cap B}\left\Vert F\left( \varepsilon
\right) \Delta G\left( \varepsilon \right) \right\Vert ,
\end{equation*}%
and Normalized Hamming quasi-metric as:%
\begin{eqnarray*}
l\left( \left( F,A\right) ,\left( G,B\right) \right) &=&\frac{\left\Vert
A\Delta B\right\Vert }{\left\Vert A\cup B\right\Vert }+\sum_{\varepsilon \in
A\cap B}\chi \left( \varepsilon \right) \\
\text{where }\chi \left( \varepsilon \right) &=&\left\{ 
\begin{tabular}{ll}
$\frac{\left\Vert F\left( \varepsilon \right) \Delta G\left( \varepsilon
\right) \right\Vert }{\left\Vert F\left( \varepsilon \right) \cup G\left(
\varepsilon \right) \right\Vert },$ & if $F\left( \varepsilon \right) \cup
G\left( \varepsilon \right) \neq \phi $ \\ 
&  \\ 
\multicolumn{1}{c}{$0,$} & otherwise.%
\end{tabular}%
\right.
\end{eqnarray*}
\end{definition}

\begin{definition}
For two soft sets $\left( F,A\right) $ and $\left( G,B\right) $ in a soft
space $\left( X,E\right) $, we define Cardinality semi-metric as: 
\begin{equation*}
c\left( \left( F,A\right) ,\left( G,B\right) \right) =\left\vert \left\Vert
A\right\Vert -\left\Vert B\right\Vert \right\vert +\sum_{\varepsilon \in
A\cap B}\left\vert \left\Vert F\left( \varepsilon \right) \right\Vert
-\left\Vert G\left( \varepsilon \right) \right\Vert \right\vert ,
\end{equation*}%
and Normalized Cardinality semi-metric as:%
\begin{equation*}
p\left( \left( F,A\right) ,\left( G,B\right) \right) =\frac{\left\vert
\left\Vert A\right\Vert -\left\Vert B\right\Vert \right\vert }{\left\Vert
E\right\Vert }+\sum_{\varepsilon \in A\cap B}\frac{\left\vert \left\Vert
F\left( \varepsilon \right) \right\Vert -\left\Vert G\left( \varepsilon
\right) \right\Vert \right\vert }{\left\Vert X\right\Vert }.
\end{equation*}
\end{definition}

Following example shows that $d$ and $l$ are quasi-metrics and $c$ and $p$
are semi-metrics, only:

\begin{example}
Let $\left( X,E\right) $ be a soft space with $X=\left\{ a,b,c,d\right\} $
and $E=\left\{ e_{1},e_{2},e_{3},e_{4}\right\} .$ We choose following soft
sets in $\left( X,E\right) :$ 
\begin{eqnarray*}
\left( F,A\right) &=&\left\{
e_{1}=\{c,b\},e_{2}=\{b\},e_{3}=\{a,b,c\},e_{4}=\{d\}\right\} , \\
\left( G,B\right) &=&\left\{ e_{2}=\{b,c\},e_{3}=\{a,b,c,d\}\right\} , \\
\left( H,C\right) &=&\left\{
e_{1}=\{b,d\},e_{2}=\{b,c,d\},e_{3}=\{a,d\},e_{4}=\{a,b,c,d\}\right\} .
\end{eqnarray*}%
Then calculations give%
\begin{equation*}
\begin{tabular}{lll}
$d\left( \left( F,A\right) ,\left( G,B\right) \right) =4$ & $d\left( \left(
G,B\right) ,\left( H,C\right) \right) =5$ & $d\left( \left( F,A\right)
,\left( H,C\right) \right) =10$ \\ 
$l\left( \left( F,A\right) ,\left( G,B\right) \right) =\frac{5}{4}$ & $%
l\left( \left( G,B\right) ,\left( H,C\right) \right) =\frac{4}{3}$ & $%
l\left( \left( F,A\right) ,\left( H,C\right) \right) =\frac{17}{6}$%
\end{tabular}%
\end{equation*}%
and hence%
\begin{eqnarray*}
d\left( \left( F,A\right) ,\left( G,B\right) \right) +d\left( \left(
G,B\right) ,\left( H,C\right) \right) &=&4+5\not\geq 10=d\left( \left(
F,A\right) ,\left( H,C\right) \right) , \\
l\left( \left( F,A\right) ,\left( G,B\right) \right) +l\left( \left(
G,B\right) ,\left( H,C\right) \right) &=&\frac{5}{4}+\frac{4}{3}\not\geq 
\frac{17}{6}=l\left( \left( F,A\right) ,\left( H,C\right) \right) .
\end{eqnarray*}%
Again choose $X=\left\{ a,b,c\right\} $, $E=\left\{
e_{1},e_{2},e_{3},e_{4}\right\} \ $and soft sets: 
\begin{eqnarray*}
\left( F,A\right) &=&\left\{
e_{3}=\{c\},e_{4}=\{a\},e_{1}=\{c,b,a\},e_{2}=\{b,a\}\right\} , \\
\left( G,B\right) &=&\left\{
e_{3}=\{c,b\},e_{4}=\{c,b\},e_{2}=\{b,a\}\right\} , \\
\left( H,C\right) &=&\left\{
e_{4}=\{c,b,a\},e_{1}=\{c\},e_{2}=\{b,a\}\right\} .
\end{eqnarray*}%
Then we get%
\begin{eqnarray*}
c\left( \left( F,A\right) ,\left( G,B\right) \right) +c\left( \left(
G,B\right) ,\left( H,C\right) \right) &=&3+1\not\geq 5=c\left( \left(
F,A\right) ,\left( H,C\right) \right) , \\
p\left( \left( F,A\right) ,\left( G,B\right) \right) +p\left( \left(
G,B\right) ,\left( H,C\right) \right) &=&\frac{11}{12}+\frac{1}{3}\not\geq 
\frac{19}{12}=p\left( \left( F,A\right) ,\left( H,C\right) \right) .
\end{eqnarray*}%
Moreover $c$ and $p$ fail to satisfy%
\begin{eqnarray*}
c\left( \left( F,A\right) ,\left( G,B\right) \right) &=&0\iff \left(
F,A\right) =\left( G,B\right) , \\
p\left( \left( F,A\right) ,\left( G,B\right) \right) &=&0\iff \left(
F,A\right) =\left( G,B\right) .
\end{eqnarray*}%
For this choose%
\begin{eqnarray*}
\left( F,A\right) &=&\left\{ e_{4}=\{b,a\},e_{1}=\{\}\right\} , \\
\left( G,B\right) &=&\left\{ e_{3}=\{b\},e_{4}=\{c,b\}\right\} ,
\end{eqnarray*}%
then%
\begin{equation*}
c\left( \left( F,A\right) ,\left( G,B\right) \right) =0=p\left( \left(
F,A\right) ,\left( G,B\right) \right) .
\end{equation*}
\end{example}

\begin{definition}
\label{my-distances}For two soft sets $\left( F,A\right) $ and $\left(
G,B\right) $ in a soft space $\left( X,E\right) ,$ where $A$ and $B$ are not
identically void$,$ we define Euclidean distance as:%
\begin{equation*}
e\left( \left( F,A\right) ,\left( G,B\right) \right) =\left\Vert A\Delta
B\right\Vert +\sqrt{\sum_{\varepsilon \in A\cap B}\left\Vert F\left(
\varepsilon \right) \Delta G\left( \varepsilon \right) \right\Vert ^{2}},
\end{equation*}%
Normalized Euclidean distance as:%
\begin{eqnarray*}
q\left( \left( F,A\right) ,\left( G,B\right) \right) &=&\frac{\left\Vert
A\Delta B\right\Vert }{\sqrt{\left\Vert A\cup B\right\Vert }}+\sqrt{%
\sum_{\varepsilon \in A\cap B}\chi \left( \varepsilon \right) }. \\
\text{where }\chi \left( \varepsilon \right) &=&\left\{ 
\begin{tabular}{ll}
$\frac{\left\Vert F\left( \varepsilon \right) \Delta G\left( \varepsilon
\right) \right\Vert ^{2}}{\left\Vert F\left( \varepsilon \right) \cup
G\left( \varepsilon \right) \right\Vert },$ & if $F\left( \varepsilon
\right) \cup G\left( \varepsilon \right) \neq \phi $ \\ 
&  \\ 
\multicolumn{1}{c}{$0,$} & otherwise%
\end{tabular}%
\right. ,
\end{eqnarray*}%
where all the radicals yield non-negative values only.
\end{definition}

\begin{proposition}
The mappings $e,q:\left( X,E\right) \times \left( X,E\right) \rightarrow 
\mathbb{R}^{+},$ as defined above, are metrics.\bigskip
\end{proposition}

\begin{lemma}
\label{lm-myDistances}For the soft sets $\left( F_{\phi },E\right) $,$\left(
F_{X},E\right) $ and an arbitrary soft set $\left( F,A\right) $ in a soft
space $\left( X,E\right) $, we have:

$\left( 1\right) $ $e\left( \left( F,A\right) ,\left( F,A\right) ^{c}\right)
=2\left\Vert A\right\Vert ,$

$\left( 2\right) $ $q\left( \left( F,A\right) ,\left( F,A\right) ^{c}\right)
=\sqrt{2\left\Vert A\right\Vert },$

$\left( 3\right) $ $e\left( \left( F_{\phi },E\right) ,\left( F_{X},E\right)
\right) =\sqrt{\left\Vert E\right\Vert \left\Vert X\right\Vert },$

$\left( 4\right) $ $q\left( \left( F_{\phi },E\right) ,\left( F_{X},E\right)
\right) =\sqrt{\left\Vert E\right\Vert }.\bigskip $
\end{lemma}

\section{\textbf{Some New Similarity Measures\protect\bigskip }}

\begin{definition}
A mapping $S:\left( X,E\right) \times \left( X,E\right) \rightarrow \left[
0,1\right] $ is said to be similarity measure if its value $S\left( \left(
F,A\right) ,\left( G,B\right) \right) ,$ for arbitrary soft sets $\left(
F,A\right) $ and $\left( G,B\right) $ in the soft space $\left( X,E\right) $%
, satisfies following axioms:

\begin{description}
\item[$\left( s1\right) $] $0\leq S\left( \left( F,A\right) ,\left(
G,B\right) \right) \leq 1,$

\item[$\left( s2\right) $] if $\left( F,A\right) =\left( G,B\right) ,$ then $%
S\left( \left( F,A\right) ,\left( G,B\right) \right) =1,$

\item[$\left( s3\right) $] $S\left( \left( F,A\right) ,\left( G,B\right)
\right) =S\left( \left( G,B\right) ,\left( F,A\right) \right) ,$

\item[$\left( s4\right) $] if $\left( F,A\right) ~\widetilde{\subseteq }%
~\left( G,B\right) $ and $\left( G,B\right) ~\widetilde{\subseteq }~\left(
H,C\right) ,$ then $S\left( \left( F,A\right) ,\left( H,C\right) \right)
\leq S\left( \left( F,A\right) ,\left( G,B\right) \right) $ and $S\left(
\left( F,A\right) ,\left( H,C\right) \right) \leq S\left( \left( G,B\right)
,\left( H,C\right) \right) .$
\end{description}
\end{definition}

\begin{definition}
For two soft sets $\left( F,A\right) $ and $\left( G,B\right) $ in a soft
space $\left( X,E\right) ,$ we define a set theoretic matching function
similarity measure as:%
\begin{equation*}
M\left( \left( F,A\right) ,\left( G,B\right) \right) =\frac{\left\Vert A\cap
B\right\Vert }{\max \left( \left\Vert A\right\Vert ,\left\Vert B\right\Vert
\right) }+\frac{\underset{\varepsilon \in A\cap B}{\dsum }\left\Vert F\left(
\varepsilon \right) \cap G\left( \varepsilon \right) \right\Vert }{\underset{%
\varepsilon \in A\cap B}{\dsum }\max \left( \left\Vert F\left( \varepsilon
\right) \right\Vert ,\left\Vert G\left( \varepsilon \right) \right\Vert
\right) }.
\end{equation*}
\end{definition}

\begin{proposition}
For the soft sets $\left( F_{\phi },E\right) $, $\left( F_{X},E\right) $ and
an arbitrary soft set $\left( F,A\right) $ in a soft space $\left(
X,E\right) $, we have:

$\left( 1\right) $ $M\left( \left( F,A\right) ,\left( F,A\right) ^{c}\right)
=0,$

$\left( 2\right) $ $M\left( \left( F_{\phi },E\right) ,\left( F_{X},E\right)
\right) =1.$
\end{proposition}

Based upon distances, defined in last section (Definition \ref{my-distances}%
), two similarity measure may be introduced, following Koczy \cite{xx00koc},
as:%
\begin{eqnarray*}
S_{K}^{e}\left( \left( F,A\right) ,\left( G,B\right) \right) &=&\frac{1}{%
1+e\left( \left( F,A\right) ,\left( G,B\right) \right) }, \\
S_{K}^{q}\left( \left( F,A\right) ,\left( G,B\right) \right) &=&\frac{1}{%
1+q\left( \left( F,A\right) ,\left( G,B\right) \right) }.
\end{eqnarray*}%
Using the definition of Williams and Steele \cite{xx00wil} we may define
another pair of similarity measures as:%
\begin{eqnarray*}
S_{W}^{e}\left( \left( F,A\right) ,\left( G,B\right) \right) &=&e^{-\alpha
\cdot e\left( \left( F,A\right) ,\left( G,B\right) \right) }, \\
S_{W}^{q}\left( \left( F,A\right) ,\left( G,B\right) \right) &=&e^{-\alpha
\cdot q\left( \left( F,A\right) ,\left( G,B\right) \right) }.
\end{eqnarray*}%
where $\alpha $ is a positive real number (parameter) called the steepness
measure.

\begin{definition}
\cite{sfst08maj} Two soft sets $\left( F,A\right) $ and $\left( G,B\right) $
in a soft space $\left( X,E\right) $ are said to be $\alpha $-similar,
denoted as $\left( F,A\right) \overset{\alpha }{\thickapprox }\left(
G,B\right) $, if 
\begin{equation*}
S\left( \left( F,A\right) ,\left( G,B\right) \right) \geq \alpha \text{ for }%
\alpha \in \left( 0,1\right) ,
\end{equation*}%
where $S$ is a similarity measure.
\end{definition}

\begin{proposition}
$\overset{\alpha }{\thickapprox }$ is reflexive and symmetric.
\end{proposition}

Majumdar and Samanta \cite{sfst08maj} have defined the notion of significant
similarity as follows:

\begin{definition}
\cite{sfst08maj} Two soft sets $\left( F,A\right) $ and $\left( G,B\right) $
in a soft space $\left( X,E\right) $ are said to be significantly similar
with respect to the similarity measure $S,$ if $S\left( \left( F,A\right)
,\left( G,B\right) \right) \geq \frac{1}{2}.$
\end{definition}

In the following example we show that two clearly non-similar soft sets come
out to be significantly similar using a Majumdar-Samant similarity measure.
But the same soft sets are rightly discerned as non-significantly similar by
a similarity measure proposed in this work:

\begin{example}
\label{ex_showing_superiority}Let $X=\left\{ a,b,c,d\right\} $ and $%
E=\left\{ e_{1},e_{2},e_{3}\right\} $ and 
\begin{eqnarray*}
\left( F,A\right) &=&\left\{ e_{1}=\left\{ {}\right\} ,e_{2}=\left\{
{}\right\} \right\} , \\
\left( G,B\right) &=&\left\{ e_{2}=\left\{ b,d\right\} \right\} , \\
\left( H,C\right) &=&\left\{ e_{2}=\left\{ b,c,d\right\} \right\} .
\end{eqnarray*}%
It is intuitively clear that $\left( F,A\right) $ and $\left( H,C\right) $
are not similar but $\left( G,B\right) $ and $\left( H,C\right) $ appear to
be considerably similar. We calculate the similarity of both pairs of soft
sets using a Majumdar-Samanta similarity measure $S^{\prime }\left(
F_{1},F_{2}\right) =\frac{1}{1+E^{S}\left( F_{1},F_{2}\right) },$ as follows:%
\begin{eqnarray*}
S^{\prime }\left( \left( F,A\right) ,\left( H,C\right) \right) &=&\frac{1}{2}%
=0.5, \\
S^{\prime }\left( \left( G,B\right) ,\left( H,C\right) \right) &=&\frac{1}{1+%
\frac{\sqrt{3}}{3}}=0.633.
\end{eqnarray*}%
Hence according to $S^{\prime }$ both the soft sets $\left( F,A\right) $ and 
$\left( G,B\right) $ are significantly similar to $\left( H,C\right) ,$
though this conclusion is counter-intuitive. On the other hand using $%
S_{K}^{e}$ (proposed in this work) we calculate similarities as:%
\begin{eqnarray*}
S_{K}^{e}\left( \left( F,A\right) ,\left( H,C\right) \right) &=&\frac{1}{2+%
\sqrt{3}}=0.25\text{,} \\
S_{K}^{e}\left( \left( G,B\right) ,\left( H,C\right) \right) &=&\frac{1}{2}%
=0.5.
\end{eqnarray*}%
Clearly $S_{K}^{e}$ has rightly discerned $\left( F,A\right) ,\left(
H,C\right) $ to be non-significantly similar and $\left( G,B\right) ,\left(
H,C\right) $ as significantly similar.
\end{example}

We now give some interesting properties of the newly introduced similarity
measures in the form of following two propositions. Proofs of these
propositions are straightforward in view of Lemma \ref{lm-myDistances}:

\begin{proposition}
For an arbitrary soft set $\left( F,A\right) $ in a soft space $\left(
X,E\right) $, we have:

$\left( 1\right) $ $S_{K}^{e}\left( \left( F,A\right) ,\left( F,A\right)
^{c}\right) =\frac{1}{1+2\left\Vert A\right\Vert },$

$\left( 2\right) $ $S_{K}^{q}\left( \left( F,A\right) ,\left( F,A\right)
^{c}\right) =\frac{1}{1+\sqrt{2\left\Vert A\right\Vert }},$

$\left( 3\right) $ $S_{W}^{e}\left( \left( F,A\right) ,\left( F,A\right)
^{c}\right) =e^{-2\left\Vert A\right\Vert \alpha },$

$\left( 4\right) $ $S_{W}^{q}\left( \left( F,A\right) ,\left( F,A\right)
^{c}\right) =e^{-\sqrt{2\left\Vert A\right\Vert }\alpha }.\bigskip $
\end{proposition}

\begin{proposition}
For the soft sets $\left( F_{\phi },E\right) $, $\left( F_{X},E\right) $ in
a soft space $\left( X,E\right) $, we have:

$\left( 1\right) $ $S_{K}^{e}\left( \left( \left( F_{\phi },E\right) ,\left(
F_{X},E\right) \right) \right) =\frac{1}{1+\sqrt{\left\Vert E\right\Vert
\left\Vert X\right\Vert }},$

$\left( 2\right) $ $S_{K}^{q}\left( \left( \left( F_{\phi },E\right) ,\left(
F_{X},E\right) \right) \right) =\frac{1}{1+\sqrt{\left\Vert E\right\Vert }},$

$\left( 3\right) $ $S_{W}^{e}\left( \left( \left( F_{\phi },E\right) ,\left(
F_{X},E\right) \right) \right) =e^{-\sqrt{\left\Vert E\right\Vert \left\Vert
X\right\Vert }\alpha },$

$\left( 4\right) $ $S_{W}^{q}\left( \left( \left( F_{\phi },E\right) ,\left(
F_{X},E\right) \right) \right) =e^{-\sqrt{\left\Vert E\right\Vert }\alpha
}.\bigskip $
\end{proposition}

\section{\textbf{An Application of Similarity Measures in Financial Diagnosis%
\protect\bigskip }}

We now present a financial diagnosis problem where similarity measures can
be applied.

The notion of similarity measure of two soft sets can be applied to detect
whether a firm is suffering from a certain economic syndrome or not. In the
following example, we estimate if two firms with observed profiles of
financial indicators are suffering from serious liquidity problem. Suppose
the firm profiles are given as:

\begin{description}
\item[Profile 1] The firm ABC maintains a beerish future outlook as well as
same behaviour in trading of its share prices. During last fiscal year the
profit-earning ratio continued to rise. Inflation is increasing
continuously. ABC has a low amount of paid-up capital and a similar
situation is seen in foreign direct investment flowing into ABC.

\item[Profile 2] The firm XYZ showed a fluctuating share price and hence a
varying future outlook. Like ABC profit-earning ratio remained beerish. As
both firms are in the same economy, inflation is also rising for XYZ and may
be considered even high in view of XYZ. Competition in the business area of
XYZ is increasing. Debit level went high but the paid-up capital lowered.
\end{description}

For this, we first construct a model soft set for liquidity-problem and the
soft sets for the firm profiles. Next we find the similarity measure of
these soft sets. If they are significantly similar, then we conclude that
the firm is possibly suffering from liquidity problem.

Let $X=\{$inflation, profit-earning ratio, share price, paid-up capital,
competitiveness, business diversification, future outlook, debt level,
foreign dirct investment, fixed income$\}$ be the collection of financial
indicators which are given in both profiles. Further let $E=\left\{ \text{%
fluctuating, medium, rising, high, beerish}\right\} $ be the universe of
parameters, which are basically linguistic labels commonly used to describe
the state of financial indicators.

The profile of a firm by observing its financial indicators may easily be
coded into a soft set using appropriate linguistic labels. Let $\left(
F,A\right) $ and $\left( G,B\right) $ be soft sets coding profiles of firms
ABC and XYZ, respectively, and are given as:%
\begin{eqnarray*}
\left( F,A\right) &=&\left\{ 
\begin{array}{c}
\text{bearish}=\left\{ \text{future outlook, share price}\right\} ,~\text{%
rising}=\left\{ \text{profit earning ratio, inflation}\right\} ,\text{ } \\ 
\text{low}=\left\{ \text{paid-up capital, foreign direct investment}\right\}%
\end{array}%
\right\} , \\
\left( G,B\right) &=&\left\{ 
\begin{array}{c}
\text{fluctuating=}\left\{ \text{share price, future outlook}\right\} ,\text{
beerish=}\left\{ \text{profit earning ratio}\right\} ,\text{ } \\ 
\text{rising=}\left\{ \text{inflation, compitition}\right\} ,\text{high=}%
\left\{ \text{inflation, debit level}\right\} ,\text{ low=}\left\{ \text{%
paid-up capital}\right\} \text{ }%
\end{array}%
\right\} .
\end{eqnarray*}%
The model soft set for a firm suffering from liquidity problem can easily be
prepared in a similar manner by help of a financial expert. In our case we
take it to be as follows:%
\begin{equation*}
\left( H,C\right) =\left\{ 
\begin{array}{c}
\text{fluctuating=}\left\{ \text{share price, future outlook}\right\} ,\text{
low=}\left\{ \text{fixed income, paid-up capital}\right\} \\ 
\text{beerish=}\left\{ \text{profit earning ratio, foreign direct investment}%
\right\} ,\text{ high=}\left\{ \text{inflation},\text{ debt level}\right\}%
\end{array}%
\right\} .
\end{equation*}

For the sake of ease in mathematical manipulation we denote the indicators
and labels by symbols as follows:%
\begin{equation*}
\begin{tabular}{lll}
$i$ & $=$ & inflation \\ 
$p$ & $=$ & profit-earning ratio \\ 
$s$ & $=$ & share price \\ 
$c$ & $=$ & paid-up capital \\ 
$m$ & $=$ & competition \\ 
$d$ & $=$ & business diversification \\ 
$o$ & $=$ & future outlook \\ 
$l$ & $=$ & debt level \\ 
$f$ & $=$ & foreign dirct investment \\ 
$x$ & $=$ & fixed income%
\end{tabular}%
\text{ and }%
\begin{tabular}{lll}
$e_{1}$ & $=$ & fluctuating \\ 
$e_{2}$ & $=$ & low \\ 
$e_{3}$ & $=$ & rising \\ 
$e_{4}$ & $=$ & high \\ 
$e_{5}$ & $=$ & beerish%
\end{tabular}%
\end{equation*}

Thus we have $X=\left\{ i,p,s,c,m,d,o,l,f,x\right\} $, $E=\left\{
e_{1},e_{2},e_{3},e_{4},e_{5}\right\} $ and the soft sets of firm profiles
become:%
\begin{eqnarray*}
\left( F,A\right) &=&\left\{ e_{5}=\left\{ o,s\right\} ,e_{3}=\left\{
p,i\right\} ,e_{2}=\left\{ s,f\right\} \right\} , \\
\left( G,B\right) &=&\left\{ e_{1}=\left\{ s,o\right\} ,e_{2}=\left\{
c\right\} ,e_{3}=\left\{ i,m\right\} ,e_{4}=\left\{ i,l\right\}
,e_{5}=\left\{ p,f\right\} \right\} , \\
\left( H,C\right) &=&\left\{ e_{1}=\left\{ o,s\right\} ,e_{2}=\left\{
c\right\} ,e_{4}=\left\{ i,l\right\} ,e_{5}=\left\{ p,f\right\} \right\} .
\end{eqnarray*}%
As the calculations give:%
\begin{eqnarray*}
S_{K}^{e}\left( \left( F,A\right) ,\left( H,C\right) \right) &=&\frac{1}{4+%
\sqrt{7}}=0.15, \\
S_{K}^{e}\left( \left( G,B\right) ,\left( H,C\right) \right) &=&\frac{1}{2}%
=0.5.
\end{eqnarray*}%
Hence we conclude that the firm with profile $\left( G,B\right) $ i.e. XYZ
is suffering from a liquidity problem as its soft set profile is
significantly similar to the standard liquidity problem profile. Whereas the
firm ABC is very less likely to be suffering from the same problem.

\begin{conclusion}
Majumdar and Samanta \cite{sfst08maj} use matrix representation based
distances of soft sets to introduce matching function and distance based
similarity measures. We first give counterexamples to show that Majumdar and
Samanta's Definition 2.7 and Lemma 3.5(3) contain errors, then prove some
properties of the distances introduced by them, thus making their Lemma 4.4,
a corllary of our result.

The tacit assumption of \cite{sfst08maj} that matrix representation is a
suitable representation for mathematical manipulation of soft sets, has been
shown to be flawed, in Section 4. This raises a natural question as to what
approach be considered suitable for similarity measures of soft sets? In one
possible reply to this we introduce set operations based measures. Our
Example \ref{ex_showing_superiority} presents a case where Majumdar-Samanta
similarity measure produces an erroneous result but the measure proposed
herein decides correctly. The new similarity measures have been applied to
the problem of financial diagnosis of firms. A technique of using linguistic
labels as parameters for soft sets has been used to model natural-language
descriptions in terms of soft sets. This exhibits the rich prospects held by
Soft Set Theory as a tool for problems in social, biological and economic
systems.
\end{conclusion}

\begin{acknowledgement}
The author wishes to express his sincere gratitude to anonymous referee(s)
for valuable comments which have improved the presentation. The author is
also thankful to the area editor Prof. John Mordeson, of this journal, for
his kindness and prompt response.
\end{acknowledgement}

\end{document}